\documentclass[a4paper,11pt]{article}

\def\version{January 30, 2012}
\usepackage{latexsym,bm,amssymb}
\usepackage[dvips]{epsfig}
\usepackage{amsthm}
\usepackage{amsmath}
\usepackage{mathrsfs}
\usepackage{times}

\nonstopmode

\newcommand{\notyet}[1]{}

\DeclareSymbolFont{AMSb}{U}{msb}{m}{n}
\DeclareSymbolFontAlphabet{\mathbb}{AMSb}

\newcommand{\e}{{\bm e}}

\newcommand{\xxd}{X}

\providecommand{\ttd}{T}
\renewcommand{\ttd}{T}

\newcommand{\be}{\begin{equation}}
\newcommand{\ee}{\end{equation}}
\newcommand{\ba}{\begin{array}}
\newcommand{\ea}{\end{array}}

\newcommand{\beqn}{\begin{eqnarray}}
\newcommand{\eeqn}{\end{eqnarray}}

\newcommand{\p}{\partial}
\newcommand{\at}[1]{\vert\sb{\sb{#1}}}

\def\Re{\,{\rm Re\, }}

\providecommand{\C}{\mathbb{C}}
\renewcommand{\C}{\mathbb{C}}
\newcommand{\R}{\mathbb{R}}

\newcommand{\Z}{\mathbb{Z}}

\newcommand{\Abs}[1]{\Big\vert#1\Big\vert}
\newcommand{\abs}[1]{\vert #1 \vert}

\newcommand{\norm}[1]{\Vert #1 \Vert}

\providecommand{\ltor}[1]{
\ifnum #1=1{\it i}\else\ifnum #1=2{\it ii}\else\ifnum #1=3{\it iii}
\else\ifnum #1=4 {\it iv}\fi\fi\fi\fi
}

\DeclareMathSymbol{\varGamma}{\mathord}{letters}{"00}
\DeclareMathSymbol{\varDelta}{\mathord}{letters}{"01}
\DeclareMathSymbol{\varSigma}{\mathord}{letters}{"06}
\DeclareMathSymbol{\varPhi}{\mathord}{letters}{"08}
\DeclareMathSymbol{\varOmega}{\mathord}{letters}{"0A}

\theoremstyle{plain}
\newtheorem{theorem}{Theorem}[section]

\newtheorem{lemma}[theorem]{Lemma}

\newtheorem{assumption}[theorem]{Assumption}

\theoremstyle{remark}
\newtheorem{remark}[theorem]{Remark}

\makeatletter\@addtoreset{equation}{section}
\makeatother

\begin{document}

\title{
Well-posedness, energy and charge conservation
\\
for nonlinear wave equations in discrete space-time}

\author{{\sc Andrew Comech$\sp 1$ and Alexander Komech$\sp 2$
\footnote{ Supported in part by Alexander von
Humboldt Research Award (2006)
and by grants FWF, DFG, and RFBR.}
}
\\
{\it\small
$\sp 1$
Mathematics Department, Texas A\&M University, College Station, TX, USA}
\\
{\it\small
$\sp 2$
Faculty of Mathematics, University of Vienna, Wien A-1090, Austria}
\\
{\it\small
$\sp {1,2}$
Institute for Information Transmission Problems, Moscow 101447, Russia}
}

\date{\version}

\maketitle

\begin{abstract}
{\small
We consider the problem of discretization
for $\mathbf{U}(1)$-invariant nonlinear wave equations
in any dimension.
We show that the classical finite-difference scheme
used by Strauss and Vazquez \cite{MR0503140}
conserves the positive-definite discrete analog of the energy
if the grid ratio is $dt/dx\le 1/\sqrt{n}$,
where $dt$ and $dx$ are the mesh sizes
of the time and space variables and $n$ is the spatial dimension.
We also show that if the grid ratio is $dt/dx=1/\sqrt{n}$,
then there is the discrete analog of the charge which is conserved.

We prove the existence and uniqueness of solutions to the
discrete Cauchy problem. We use the energy conservation
to obtain the a priori bounds for finite energy solutions,
thus showing that the Strauss -- Vazquez finite-difference scheme
for the nonlinear Klein-Gordon equation with positive nonlinear
term in the Hamiltonian is conditionally stable.
}


\end{abstract}


\section{Introduction}
We study the $\mathbf{U}(1)$-invariant nonlinear wave equation
discretized in space and time.
Our objective has been to
find a stable finite-difference scheme for numerical simulation
of the nonlinear wave processes,
which corresponds to a well-posed Cauchy problem
and provides us with the a priori energy bounds.

The discretized models are widely studied
in applied mathematics
and in theoretical physics.
Such models originally appeared in the condensed matter theory,
due to atoms in a crystal forming a lattice.
Now these models
occupy a prominent place
in theoretical physics,
in part due to some of these models
(such as the Ising model)
being exactly solvable.
Lattice models
are also used
for the description of polymers.

The paper \cite{MR0503140}
set the ground for considering the
energy-conserving difference schemes
for the nonlinear Klein-Gordon equations
and nonlinear wave equations.
The importance of having
conserved quantities in the numerical scheme
was illustrated
by noticing that instability occurs
for a finite-difference scheme
which does not conserve the energy.
The authors gave
the implicit difference scheme
and wrote down the expression for the energy
conserved by that scheme.
This finite-difference scheme
was favorably compared to three other schemes in
\cite{MR1047100}.
The higher dimensional analog of the Strauss-Vazquez
scheme
and the corresponding energy-momentum tensor
was written in \cite{MR1318634}.
The general theory
of finite-difference schemes
for the nonlinear Klein-Gordon equation
aimed at the energy conservation
was developed in 
\cite{MR1360462} and \cite{MR1852556}.
In
the paper \cite{SubCMAME} the energy preserving schemes
are constructed
for a wide class of second order nonlinear Hamiltonian systems
of wave equations.

Importance of the Strauss -- Vazquez finite-difference scheme
over schemes from \cite{MR1360462,MR1852556}
is in allowing for an easier solution algorithm.
Namely, the discrete scheme
involves the value of the unknown function
at the ``next'' moment of time
only at a single lattice point,
and can be solved (numerically)
with respect to the value at that point.
See Remark~\ref{remark-ssa} below.
At the same time,
the corresponding discrete energy for that
scheme (see \eqref{def-energy-sv} below)
contains quadratic terms which are not positive-definite,
showing that the scheme is not
\emph{unconditionally stable}.

We use the same finite-difference scheme
by Strauss and Vazquez \cite{MR0503140}.
We show that under the assumption
$dt/dx\le 1/\sqrt{n}$ on the grid ratio,
where $n$ is the number of spatial dimensions,
the expression for the conserved
discrete energy is positive-definite,
providing one with the a priori energy estimates
in the case of the discrete nonlinear Klein-Gordon equation.
Moreover, under the assumption
$dt/dx=1/\sqrt{n}$ on the grid ratio,
we show that
the equation possesses conserved discrete charge.
The continuous limits of our discrete versions for
the energy and charge coincide with the 
energy and charge in the continuous case.

Let us emphasize that
the positive definiteness of the energy
allows one to have the a priori bounds
on the norm of the solution.
Such a priori bounds are of utmost importance
for applications.
Numerically,
such bounds
indicate the stability of the finite-difference scheme.

The discrete charge conservation
does not seem to be particularly important on its own,
but could be considered as an indication that
the $\mathbf{U}(1)$-invariance of the continuous equation
is in a certain sense compatible with the chosen discretization procedure.
See the discussion in \cite[Section 1]{MR1360462}.

\medskip

In Section~\ref{sect-disc},
we formulate our 
main results.
We establish the existence
of the solutions to the corresponding discrete Cauchy problem
and analyze the uniqueness of solutions.
We show that for confining polynomial potentials
the uniqueness will follow
when the mesh size is sufficiently small.
Besides, we describe a class of polynomial nonlinearities
for which the value of the mesh size
could be readily specified.
We also consider the charge conservation.
The proofs related to the well-posedness
are in Section~\ref{sect-eu}.

\section{Main results}

\label{sect-disc}

\subsection{Continuous case}

Let us first consider the $\mathbf{U}(1)$-invariant nonlinear wave equation
\begin{equation}\label{nlw}
\ddot\psi(x,t)=\Delta\psi(x,t)
-2\p\sb\lambda v(x,\abs{\psi(x,t)}^2)\psi(x,t),
\qquad
x\in\R^n,
\end{equation}
where 
$\psi(x,t)\in\C^N$,
$N\ge 1$,
and
$v(x,\lambda)$
is such that $v\in C(\R^n\times\R)$
and $v(x,\cdot)\in C\sp{2}(\R)$ for each $x\in\R^n$.
Equation \eqref{nlw}
can be written in the Hamiltonian form,
with the Hamiltonian
\begin{equation}\label{def-energy}
\mathscr{E}(\psi,\dot\psi)
=\int\sb{\R^n}
\Big[
\frac{\abs{\dot\psi}^2}{2}+\frac{\abs{\nabla\psi}^2}{2}+v(x,\abs{\psi(x,t)}^2)
\Big]\,dx,
\end{equation}
where
for $\psi\in\C^N$ we define $\abs{\psi}^2=\bar\psi\cdot\psi$.
The value of the Hamiltonian
functional $\mathscr{E}$
and the value of the charge functional
\begin{equation}\label{def-charge}
\mathscr{Q}(\psi,\dot\psi)
=\frac{i}{2}\int\sb{\R^n}
\big(
\bar\psi\cdot\dot\psi-\dot{\bar\psi}\cdot\psi
\big)\,dx
\end{equation}
are formally conserved for solutions to \eqref{nlw}.
A particular case of \eqref{nlw}
is the nonlinear Klein-Gordon equation,
with $v(x,\lambda)=\frac{m^2}{2}\lambda+z(x,\lambda)$,
with $m>0$:
\begin{equation}\label{nlkg}
\ddot\psi=\Delta\psi-m^2\psi-2\p\sb\lambda z(x,\abs{\psi}^2)\psi,
\qquad
x\in\R^n,
\quad
t\in\R.
\end{equation}
In the case
$z(x,\lambda)\ge 0$ for all $x\in\R^n$, $\lambda\ge 0$,
the conservation of the energy
$
\int\sb{\R^n}
\Big[
\frac{\abs{\dot\psi}^2}{2}+\frac{\abs{\nabla\psi}^2}{2}+\frac{m^2\abs{\psi}^2}{2}
+z(x,\abs{\psi}^2)
\Big]\,dx
$
yields an a priori estimate
on the norm of the solution:
\begin{equation}\label{apec}
\int\sb{\R^n}
\abs{\psi(x,t)}^2\,dx
\le \frac{2}{m^2}\mathscr{E}(\psi\at{t=0},\dot\psi\at{t=0}).
\end{equation}

\subsection{Discretized equation}

Let us now describe the discretized equation.
Let $(\xxd,\ttd)\in\Z^n\times\Z$
denote a point of the space-time lattice.
We will always indicate
the temporal dependence by superscripts 
and the spatial dependence by subscripts.
Fix $\varepsilon>0$, and
let $V\sb{\xxd}(\lambda)=v(\varepsilon \xxd,\lambda)$
be a function on $\Z^n\times\R$,
so that $V\sb{\xxd}\in C\sp{2}(\R)$ for each $\xxd\in\Z^n$.
We introduce
\begin{equation}\label{vmv}
B\sb{\xxd}(\lambda,\mu):=
\left\{
\begin{array}{l}
\frac{V\sb{\xxd}(\lambda)-V\sb{\xxd}(\mu)}{\lambda-\mu},
\qquad
\lambda\ne\mu,
\\
\p\sb\lambda V\sb{\xxd}(\lambda),
\qquad
\lambda=\mu,
\end{array}
\right.
\qquad
\lambda,\,\mu\in\R,
\quad
\xxd\in\Z^n,
\end{equation}
and consider the
standard
implicit finite-difference scheme for \eqref{nlw}
\cite{MR0503140}:
\begin{equation}\label{dkg-c}
\frac{\psi\sb{\xxd}\sp{\ttd+1}-2\psi\sb{\xxd}\sp{\ttd}+\psi\sb{\xxd}\sp{\ttd-1}}{\tau^2}
=
\sum\sb{j=1}\sp{n}
\frac{\psi\sb{\xxd+\e\sb j}\sp{\ttd}-2\psi\sb{\xxd}\sp{\ttd}+\psi\sb{\xxd-\e\sb j}\sp{\ttd}}{\varepsilon^2}
-
B\sb{\xxd}(\abs{\psi\sb{\xxd}\sp{\ttd+1}}^2,\abs{\psi\sb{\xxd}\sp{\ttd-1}}^2)
(\psi\sb{\xxd}\sp{\ttd+1}+\psi\sb{\xxd}\sp{\ttd-1}),
\end{equation}
where
$\psi\sb{\xxd}\sp{\ttd}\in\C^N$, $N\ge 1$,
is
defined on the lattice $(\xxd,\ttd)\in\Z^n\times\Z$.
Above,
\begin{equation}\label{def-ej}
\e\sb 1=(1,0,0,0,\dots)\in\Z^n,
\qquad
\e\sb 2=(0,1,0,0,\dots)\in\Z^n,
\qquad
\mbox{etc.}
\end{equation}

\begin{remark}
The continuous limit of
\eqref{dkg-c}
is given by \eqref{nlw},
with $\varepsilon \xxd$
corresponding to $x\in\R^n$
and $\tau \ttd$ corresponding to $t\in\R$.
Since
$\p\sb\lambda V\sb{\xxd}(\lambda)
=B\sb{\xxd}(\lambda,\lambda)$,
the continuous limit of the last term
in the right-hand side of \eqref{dkg-c}
coincides with the right-hand side
in \eqref{nlw}.
\end{remark}

\bigskip

We assume that
$\psi\sb{\xxd}\sp{\ttd}$ takes values in $\C^N$ with $N\ge 1$.

\begin{remark}\label{remark-ssa}
An advantage of the Strauss-Vazquez finite-difference scheme \eqref{dkg-c}
over other energy-preserving schemes discussed in
\cite{MR1360462,MR1852556}
is the fact that
at the moment $\ttd+1$
the relation
\eqref{dkg-c} only involves the function $\psi$
at the point $\xxd$, allowing for a simple
realization of the solution algorithm
even in higher dimensional case.
\end{remark}

\subsection{Well-posedness}


We will denote by $\psi\sp{\ttd}$
the function $\psi$
defined on the lattice $(\xxd,\ttd)\in\Z^n\times\Z$
at the moment $\ttd\in\Z$.

\begin{theorem}[Existence of solutions]
\label{theorem-e}
Assume that
\begin{equation}\label{def-k1}
k\sb 1:=\inf\sb{\xxd\in\Z^n,\lambda\ge 0}
\p\sb\lambda V\sb{\xxd}(\lambda)>-\infty.
\end{equation}
Define
\[
\tau\sb 1
=\left\{
\begin{array}{l}
\sqrt{-1/k\sb 1},\qquad k\sb 1<0;
\\
+\infty,\qquad\quad k\sb 1\ge 0.
\end{array}
\right.
\]
Then for any $\tau\in(0,\tau\sb{1})$
and any $\varepsilon>0$
there exists a global solution
$\psi\sp{\ttd}$,
$\ttd\in\Z$,
to the Cauchy problem for equation \eqref{dkg-c}
with arbitrary initial data
$\psi\sp{0}$,
$\psi\sp{1}$
(which stand for $\psi\sp{\ttd}$ at $\ttd=0$ and $\ttd=1$).

Moreover,
if $(\psi\sp{0},\psi\sp{1})\in l\sp 2(\Z^n)\times l\sp 2(\Z^n)$,
one has
$\psi\sp{\ttd}\in l\sp 2(\Z^n)$
for all $\ttd\in\Z$.
\end{theorem}

\begin{remark}
We do not claim
in this theorem
that $\norm{\psi\sp{\ttd}}\sb{l\sp 2(\Z^n)}$
is uniformly bounded for all $\ttd\in\Z$.
For the a priori estimates on $\norm{\psi\sp{\ttd}}\sb{l\sp 2(\Z^n)}$,
see Theorem~\ref{theorem-a-priori} below.
\end{remark}

One can readily check that
any
$\xxd$-independent
polynomial potential of the form
\begin{equation}\label{poly}
V\sb{\xxd}(\lambda)
=V(\lambda)
=\sum\sb{q=0}\sp{p}C\sb{q}\lambda^{q+1},
\qquad
C\sb{q}\in\R,
\qquad
C\sb{p}>0
\end{equation}
satisfies
\eqref{def-k1}.
Note that since $\lim\sb{\lambda\to +\infty}V(\lambda)=+\infty$,
this potential is confining.

\begin{remark}
Note that in the case
of $V\sb{\xxd}(\lambda)$ given by
\eqref{poly},
by the little B\'ezout theorem,
$B\sb{\xxd}(\lambda,\mu)$
defined in \eqref{vmv}
is  a polynomial of $\lambda$ and  $\mu$
with real coefficients.
\end{remark}

\begin{theorem}[Uniqueness of solutions]
\label{theorem-u}
Assume that the functions
$
K\sb{\xxd}\sp\pm(\lambda,\mu)
=
B\sb{\xxd}(\lambda,\mu)
+2\p\sb\lambda B\sb{\xxd}(\lambda,\mu)
(\lambda\pm\sqrt{\lambda\mu})
$
are bounded from below:
\begin{equation}\label{def-k2}
k\sb 2:=
\inf\sb{\pm,\,\xxd\in\Z^n,\,
\lambda\ge 0,\,\mu\ge 0}
K\sb{\xxd}\sp\pm(\lambda,\mu)>-\infty.
\end{equation}
Define
\[
\tau\sb 2
=\left\{
\begin{array}{l}
\sqrt{-1/k\sb 2},\qquad k\sb 2<0;
\\
+\infty,\qquad\quad k\sb 2\ge 0.
\end{array}
\right.
\]
Then for any $\tau\in(0,\tau\sb{2})$
and any $\varepsilon>0$
there exists a solution to the Cauchy problem
for equation \eqref{dkg-c}
with arbitrary initial data $(\psi\sp{0},\psi\sp{1})$,
and this solution is unique.
\end{theorem}

\begin{remark}\label{remark-k2-k3}
Note that since
\[
K\sb{\xxd}\sp{-}(\lambda,\lambda)=B\sb{\xxd}(\lambda,\lambda)
=\p\sb\lambda W\sb{\xxd}(\lambda),
\]
the values of $k\sb 1$ and $k\sb 2$ from
Theorem~\ref{theorem-e}
and
Theorem~\ref{theorem-u},
whether $k\sb 2>-\infty$,
are related by
$k\sb 2\le k\sb 1$,
and then
the values of $\tau\sb 1$ and $\tau\sb 2$
from these theorems
are related by
$\tau\sb 2\le \tau\sb 1$.
\end{remark}

\begin{theorem}[Existence and uniqueness for polynomial nonlinearities]
\label{theorem-pol}

\ \begin{enumerate}
\item
\label{theorem-pol-i}

The condition
\eqref{def-k2}
holds for any confining polynomial potential
\eqref{poly}.
\item
\label{theorem-pol-ii}
Assume that
\begin{equation}\label{w4}
V\sb{\xxd}(\lambda)=\sum\sb{q=0}\sp{4}C\sb{\xxd,q}\lambda^{q+1},
\qquad
\xxd\in\Z^n,\quad\lambda\ge 0,
\end{equation}
where
$C\sb{\xxd,q}\ge 0$ for $\xxd\in\Z^n$ and $1\le q\le 4$,
and $C\sb{\xxd,0}$ are uniformly bounded from below:
\begin{equation}\label{def-k3}
k\sb 3:=
\inf\sb{\xxd\in\Z^n}
C\sb{\xxd,0}>-\infty.
\end{equation}
\[
\tau\sb 3
=\left\{
\begin{array}{l}
\sqrt{-1/k\sb 3},\qquad k\sb 3<0;
\\
+\infty,\qquad\quad k\sb 3\ge 0.
\end{array}
\right.
\]
Then for any $\tau\in(0,\tau\sb{3})$
and any $\varepsilon>0$
there exists a solution to the Cauchy problem
for equation \eqref{dkg-c}
with arbitrary initial data $(\psi\sp{0},\psi\sp{1})$,
and this solution is unique.
\end{enumerate}
\end{theorem}

Thus, even though the potential \eqref{poly}
satisfies conditions
\eqref{def-k1} and \eqref{def-k2}
in Theorem~\ref{theorem-e} and Theorem~\ref{theorem-u},
the corresponding values $\tau\sb 1$ and $\tau\sb 2$
could be hard to specify explicitly.
Yet,
the second part of Theorem~\ref{theorem-pol}
gives a simple description
of a class of $\xxd$-dependent polynomials
$V\sb{\xxd}(\lambda)$
for which
the range of admissible $\tau>0$
can be readily specified.

\medskip

We will prove
existence and uniqueness results
stated in
Theorems~\ref{theorem-e}, ~\ref{theorem-u}, and~\ref{theorem-pol}
in Section~\ref{sect-eu}.

\subsection{Energy conservation}

\begin{theorem}[Energy conservation]
\label{theorem-energy}
Let $\psi$
be a solution to equation \eqref{dkg-c}
such that
$\psi\sp{\ttd}\in l\sp 2(\Z^n)$ for all $\ttd\in\Z$.
Then the discrete energy
\begin{equation}\label{def-energy-t}
E\sp{\ttd}
=
\sum\sb{\xxd\in\Z^n}
\Big[
\big(
\frac{1}{\tau^2}-\frac{n}{\varepsilon^2}
\big)
\frac{\abs{\psi\sb{\xxd}\sp{\ttd+1}-\psi\sb{\xxd}\sp{\ttd}}^2}{2}
+
\sum\sb{j=1}\sp{n}\sum\limits\sb{\pm}
\frac{\abs{\psi\sb{\xxd}\sp{\ttd+1}-\psi\sb{\xxd\pm\e\sb j}\sp{\ttd}}^2
}{4\varepsilon^2}
+
\frac{V\sb{\xxd}(\abs{\psi\sb{\xxd}\sp{\ttd+1}}^2)+V\sb{\xxd}(\abs{\psi\sb{\xxd}\sp{\ttd}}^2)}{2}
\Big]\varepsilon^n
\end{equation}
is conserved.
\end{theorem}

\begin{remark}
The discrete energy is positive-definite
if the grid ratio satisfies
\begin{equation}\label{grid-ratio-condition}
\frac{\tau}{\varepsilon}\le\frac{1}{\sqrt{n}}.
\end{equation}
\end{remark}

\begin{remark}
The continuous limit
of the discrete energy $E\sp{\ttd}$
defined in
\eqref{def-energy-t}
coincides with the energy functional
\eqref{def-energy}
of the continuous nonlinear wave equation \eqref{nlw}.
\end{remark}

\begin{remark}
If $\psi\sp 0$ and $\psi\sp 1\in l\sp 2(\Z^n)$,
then, by Theorem~\ref{theorem-e},
one also has $\psi\sp{\ttd}\in l\sp 2(\Z^n)$
for all $\ttd\in\Z$
as long as
\[
\inf\sb{\xxd\in\Z^n,\,\lambda\ge 0}
\p\sb\lambda V\sb{\xxd}(\lambda)>-\infty.
\]
\end{remark}

\begin{proof}
For any $u$, $v\in\C^N$,
there is the identity
\begin{equation}\label{kk}
\abs{u}^2-\abs{v}^2
=
\Re
\left[(\bar u-\bar v)\cdot(u+v)\right].
\end{equation}
Applying \eqref{kk}, one has:
\begin{equation}\label{kkk}
\sum\sb{\xxd\in\Z^n}
\big(
\abs{\psi\sb{\xxd}\sp{\ttd+1}-\psi\sb{\xxd}\sp{\ttd}}^2
-\abs{\psi\sb{\xxd}\sp{\ttd}-\psi\sb{\xxd}\sp{\ttd-1}}^2
\big)
=
\Re
\sum\sb{\xxd\in\Z^n}
\big(
\bar\psi\sb{\xxd}\sp{\ttd+1}-\bar\psi\sb{\xxd}\sp{\ttd-1}
\big)
\cdot
\big(
\psi\sb{\xxd}\sp{\ttd+1}-2\psi\sb{\xxd}\sp{\ttd}
+\psi\sb{\xxd}\sp{\ttd-1}
\big).
\end{equation}
Using \eqref{kk},
we also derive the following identity
for any function
$\psi\sb{\xxd}\sp{\ttd}\in\C^N$:
\begin{eqnarray}\label{pmpm}
&&
\hskip -10pt
\sum\sb{\xxd\in\Z^n}
\sum\sb{j=1}\sp{n}
\Big[
\abs{\psi\sb{\xxd}\sp{\ttd+1}-\psi\sb{\xxd-\e\sb j}\sp{\ttd}}^2
-\abs{\psi\sb{\xxd-\e\sb j}\sp{\ttd}-\psi\sb{\xxd}\sp{\ttd-1}}^2
+\abs{\psi\sb{\xxd}\sp{\ttd+1}-\psi\sb{\xxd+\e\sb j}\sp{\ttd}}^2
-\abs{\psi\sb{\xxd+\e\sb j}\sp{\ttd}-\psi\sb{\xxd}\sp{\ttd-1}}^2
\Big]
\nonumber
\\
&&
\hskip -10pt
=\Re\sum\sb{\xxd\in\Z^n}
\sum\sb{j=1}\sp{n}
\Big[
(\bar\psi\sb{\xxd}\sp{\ttd+1}-\bar\psi\sb{\xxd}\sp{\ttd-1})
\cdot
(\psi\sb{\xxd}\sp{\ttd+1}-2\psi\sb{\xxd\pm\e\sb j}\sp{\ttd}+\psi\sb{\xxd}\sp{\ttd-1})
+
(\bar\psi\sb{\xxd}\sp{\ttd+1}-\bar\psi\sb{\xxd}\sp{\ttd-1})
\cdot
(\psi\sb{\xxd}\sp{\ttd+1}-2\psi\sb{\xxd+\e\sb j}\sp{\ttd}+\psi\sb{\xxd}\sp{\ttd-1})
\Big]
\nonumber
\\
&&
\hskip -10pt
=\Re\sum\sb{\xxd\in\Z^n}
(\bar\psi\sb{\xxd}\sp{\ttd+1}-\bar\psi\sb{\xxd}\sp{\ttd-1})
\cdot
\Big[
2n
\big(\psi\sb{\xxd}\sp{\ttd+1}
-2\psi\sb{\xxd}\sp{\ttd}
+\psi\sb{\xxd}\sp{\ttd-1}
\big)
-
2
\sum\sb{j=1}\sp{n}
\big(\psi\sb{\xxd+\e\sb j}\sp{\ttd}
-2\psi\sb{\xxd}\sp{\ttd}
+\psi\sb{\xxd-\e\sb j}\sp{\ttd}
\big)
\Big].
\end{eqnarray}
Further, \eqref{vmv}
together with \eqref{kk}
imply that
\begin{equation}\label{vmvmv}
V\sb{\xxd}(\abs{\psi\sb{\xxd}\sp{\ttd+1}}^2)-V\sb{\xxd}(\abs{\psi\sb{\xxd}\sp{\ttd-1}}^2)
=
\Re
\big[(\bar\psi\sb{\xxd}\sp{\ttd+1}-\bar\psi\sb{\xxd}\sp{\ttd-1})
\cdot(\psi\sb{\xxd}\sp{\ttd+1}+\psi\sb{\xxd}\sp{\ttd-1})
\big]
B\sb{\xxd}(\abs{\psi\sb{\xxd}\sp{\ttd+1}}^2,\abs{\psi\sb{\xxd}\sp{\ttd-1}}^2).
\end{equation}
Taking into account \eqref{kkk}, \eqref{pmpm}, and \eqref{vmvmv},
we compute:
\begin{eqnarray}
&&
\frac{E\sp{\ttd}-E\sp{\ttd-1}}{\varepsilon^n}
=
\sum\sb{\xxd\in\Z^n}
\Big[
\Big(
\frac{1}{\tau^2}-\frac{n}{\varepsilon^2}
\Big)
\frac{
\abs{\psi\sb{\xxd}\sp{\ttd+1}-\psi\sb{\xxd}\sp{\ttd}}^2
-\abs{\psi\sb{\xxd}\sp{\ttd}-\psi\sb{\xxd}\sp{\ttd-1}}^2
}{2}
\nonumber\\
&&
+
\sum\sb{j=1}\sp{n}
\sum\limits\sb{\pm}
\frac{
\abs{\psi\sb{\xxd}\sp{\ttd+1}-\psi\sb{\xxd\pm\e\sb j}\sp{\ttd}}^2
-\abs{\psi\sb{\xxd\pm\e\sb j}\sp{\ttd}-\psi\sb{\xxd}\sp{\ttd-1}}^2
}{4\varepsilon^2}
+
\frac{V\sb{\xxd}(\abs{\psi\sb{\xxd}\sp{\ttd+1}}^2)-V\sb{\xxd}(\abs{\psi\sb{\xxd}\sp{\ttd-1}}^2)}{2}
\Big]
\nonumber\\
&&
=
\Re\sum\sb{\xxd\in\Z^n}
(\bar\psi\sb{\xxd}\sp{\ttd+1}-\bar\psi\sb{\xxd}\sp{\ttd-1})
\cdot\Big[
\Big(
\frac{1}{\tau^2}-\frac{n}{\varepsilon^2}
\Big)
\frac{\psi\sb{\xxd}\sp{\ttd+1}-2\psi\sb{\xxd}\sp{\ttd}+\psi\sb{\xxd}\sp{\ttd-1}}{2}
+
\nonumber\\
&&
+
\frac{
n
\big(\psi\sb{\xxd}\sp{\ttd+1}-2\psi\sb{\xxd}\sp{\ttd}+\psi\sb{\xxd}\sp{\ttd-1}\big)
-
\sum\limits\sb{j=1}\sp{n}
\big(\psi\sb{\xxd+\e\sb j}\sp{\ttd}
-2\psi\sb{\xxd}\sp{\ttd}
+\psi\sb{\xxd-\e\sb j}\sp{\ttd}
\big)
}{2\varepsilon^2}
+
\frac{\psi\sb{\xxd}\sp{\ttd+1}+\psi\sb{\xxd}\sp{\ttd-1}}{2}
B\sb{\xxd}(\abs{\psi\sb{\xxd}\sp{\ttd+1}}^2,\abs{\psi\sb{\xxd}\sp{\ttd-1}}^2)
\Big]
.
\nonumber
\end{eqnarray}
The expression in the square brackets
adds up to zero due to \eqref{dkg-c}.
It follows that $E\sp{\ttd}=E\sp{\ttd-1}$
for all $\ttd\in\Z$.
\end{proof}

\subsection{A priori estimates}

\begin{theorem}[A priori estimates]
\label{theorem-a-priori}
Assume that $\varepsilon>0$ and $\tau>0$ satisfy
\[
\frac{\tau}{\varepsilon}\le\frac{1}{\sqrt{n}}.
\]
Assume that
\begin{equation}\label{w-nlkg}
V\sb{\xxd}(\lambda)=\frac{m^2}{2}\lambda+W\sb{\xxd}(\lambda),
\end{equation}
where $m>0$,
and for each $\xxd\in\Z^n$
the function
$W\sb{\xxd}\in C\sp{2}(\R)$
satisfies
$W\sb{\xxd}(\lambda)\ge 0$
for $\lambda\ge 0$.
Then any solution $\psi\sb{\xxd}\sp{\ttd}$ to the Cauchy problem
\eqref{dkg-c}
with arbitrary initial data
$(\psi\sp{0},\psi\sp{1})\in l\sp 2(\Z^n)\times l\sp 2(\Z^n)$
satisfies the \emph{a priori} estimate
\begin{equation}\label{ape}
\norm{\psi\sp{\ttd}}\sb{l\sp 2}^2\varepsilon^n\le\frac{4E\sp{0}}{m^2},
\end{equation}
where $E\sp{0}$
is the energy \eqref{def-energy-t}
of the solution $\psi\sb{\xxd}\sp{\ttd}$
at the moment $\ttd=0$.
\end{theorem}

\begin{proof}
This immediately follows from the conservation
of the energy
\eqref{def-energy-t}
with $V\sb{\xxd}(\lambda)$
given by \eqref{w-nlkg},
\[
E\sp{\ttd}
=
\sum\sb{\xxd\in\Z^n}
\Big[
\Big(
\frac{1}{\tau^2}-\frac{n}{\varepsilon^2}
\Big)
\frac{\abs{\psi\sb{\xxd}\sp{\ttd+1}-\psi\sb{\xxd}\sp{\ttd}}^2}{2}
+
\sum\sb{j=1}\sp{n}
\frac{\abs{\psi\sb{\xxd}\sp{\ttd+1}-\psi\sb{\xxd-\e\sb j}\sp{\ttd}}^2
+\abs{\psi\sb{\xxd}\sp{\ttd+1}-\psi\sb{\xxd+\e\sb j}\sp{\ttd}}^2
}{4\varepsilon^2}
\]
\[
+
\frac{m^2(\abs{\psi\sb{\xxd}\sp{\ttd+1}}^2+\abs{\psi\sb{\xxd}\sp{\ttd}}^2)}{4}
+\frac{W\sb{\xxd}(\abs{\psi\sb{\xxd}\sp{\ttd+1}}^2)+W\sb{\xxd}(\abs{\psi\sb{\xxd}\sp{\ttd}}^2)}{2}
\Big]\varepsilon^n.
\]
\end{proof}

\begin{remark}
In the continuous limit
$\varepsilon\to 0$,
the relation
\eqref{ape}
is similar to the a priori estimate
\eqref{apec}
for the solutions
to the continuous nonlinear Klein-Gordon equation \eqref{nlkg}.
\end{remark}

\begin{remark}
In
\cite{MR0503140},
in the case $\psi\sb{\xxd}\sp{\ttd}\in\R$, $(\xxd,\ttd)\in\Z\times\Z$
(in the dimension $n=1$),
the following expression
for the discretized energy was introduced:
\begin{equation}\label{def-energy-sv}
E\sb{SV}\sp{\ttd}
=
\frac{1}{2}
\sum\sb{\xxd\in\Z^n}
\Big[
\frac{(\psi\sb{\xxd}\sp{\ttd+1}-\psi\sb{\xxd}\sp{\ttd})^2}{\tau^2}
+
\frac{(\psi\sb{\xxd+1}\sp{\ttd+1}-\psi\sb{\xxd}\sp{\ttd+1})(\psi\sb{\xxd+1}\sp{\ttd}-\psi\sb{\xxd}\sp{\ttd})}
{\varepsilon^2}
+V(\abs{\psi\sb{\xxd}\sp{\ttd+1}}^2)
+V(\abs{\psi\sb{\xxd}\sp{\ttd}}^2)
\Big].
\end{equation}
The presence of the second term which is
not positive-definite
deprives one of the a priori $l\sp 2$ bound
on $\psi$,
such as the one
stated in Theorem~\ref{theorem-a-priori}.
In view of this,
the Strauss-Vazquez finite-difference scheme
for the nonlinear Klein-Gordon equation
is not \emph{unconditionally stable}.
Other schemes
(conditionally and unconditionally stable)
were proposed in \cite{MR1360462,MR1852556}.
Now,
due to the a priori bound \eqref{ape},
we deduce that, as the matter of fact,
the Strauss-Vazquez scheme
is stable
in $n$ dimensions
under the condition that the grid ratio is
$\tau/\varepsilon\le 1/\sqrt{n}$.
Note that in the case $\psi\in\R$,
the Strauss-Vazquez energy \eqref{def-energy-sv}
agrees with the energy defined in \eqref{def-energy-t}.
\end{remark}

\subsection{The charge conservation}

Let us consider the charge conservation.
We will define the discrete charge
under the following assumption:

\begin{assumption}\label{ass-alpha-n}
\begin{equation}\label{alpha-n}
\frac{\tau}{\varepsilon}
=\frac{1}{\sqrt{n}}.
\end{equation}
\end{assumption}
Under Assumption~\ref{ass-alpha-n},
$\psi\sb{\xxd}\sp{\ttd}$ drops out of equation \eqref{dkg-c};
the latter can be written as
\begin{equation}\label{dkg-fn}
\big(\psi\sb{\xxd}\sp{\ttd+1}+\psi\sb{\xxd}\sp{\ttd-1}\big)
\big(
1+
\tau^2
B\sb{\xxd}(\abs{\psi\sb{\xxd}\sp{\ttd+1}}^2,\abs{\psi\sb{\xxd}\sp{\ttd-1}}^2)
\big)
=
\frac{1}{n}
\sum\sb{j=1}\sp{n}
(\psi\sb{\xxd+\e\sb j}\sp{\ttd}+\psi\sb{\xxd-\e\sb j}\sp{\ttd}).
\end{equation}

\begin{theorem}[Charge conservation]
\label{theorem-charge}
Let Assumption~\ref{ass-alpha-n} be satisfied.
Let $\psi$ be a solution to equation \eqref{dkg-fn}
such that
$\psi\sp{\ttd}\in l\sp 2(\Z^n)$ for all $\ttd\in\Z$
(see Theorem~\ref{theorem-e}).
Then the discrete charge
\begin{equation}\label{def-charge-t}
Q\sp{\ttd}
=
\frac{i}{4\tau}
\sum\sb{\xxd\in\Z^n}
\big[
\bar\psi\sb{\xxd+\e\sb j}\sp{\ttd}\cdot\psi\sb{\xxd}\sp{\ttd+1}
+\bar\psi\sb{\xxd-\e\sb j}\sp{\ttd}\cdot\psi\sb{\xxd}\sp{\ttd+1}
-\bar\psi\sb{\xxd}\sp{\ttd+1}\cdot\psi\sb{\xxd+\e\sb j}\sp{\ttd}
-\bar\psi\sb{\xxd}\sp{\ttd+1}\cdot\psi\sb{\xxd-\e\sb j}\sp{\ttd}
\big]
\varepsilon^n
\end{equation}
is conserved.
\end{theorem}

\begin{remark}
The continuous limit
of the discrete charge $Q$
defined in
\eqref{def-charge-t}
coincides with the charge functional
\eqref{def-charge}
of the continuous nonlinear wave equation \eqref{nlw}.
\end{remark}

\begin{proof}
Let us prove the charge conservation.
One has:
\[
\frac{4\,\tau}{i\varepsilon^n}Q\sp{\ttd}
=\sum\sb{\xxd\in\Z^n}
\sum\sb{j=1}\sp{n}
\sum\sb{\pm}
\Big[
\bar\psi\sb{\xxd\pm\e\sb j}\sp{\ttd}\cdot\psi\sb{\xxd}\sp{\ttd+1}
-\mathrm{C.\,C.\,}
\Big],
\]
\[
\frac{4\,\tau}{i\varepsilon^n}Q\sp{\ttd-1}
=
\sum\sb{\xxd\in\Z^n}
\sum\sb{j=1}\sp{n}
\sum\sb{\pm}
\Big[
\bar\psi\sb{\xxd\pm\e\sb j}\sp{\ttd-1}\cdot\psi\sb{\xxd}\sp{\ttd}
-\mathrm{C.\,C.\,}
\Big]
=
-
\sum\sb{\xxd\in\Z^n}
\sum\sb{j=1}\sp{n}
\sum\sb{\pm}
\Big[
\bar\psi\sb{\xxd\pm\e\sb j}\sp{\ttd}\cdot\psi\sb{\xxd}\sp{\ttd-1}
-\mathrm{C.\,C.\,}
\Big].
\]
Therefore,
\begin{eqnarray}
\frac{4\,\tau\big(Q\sp{\ttd}-Q\sp{\ttd-1}\big)}{i\varepsilon^n}
&=&
\sum\sb{\xxd\in\Z^n}
\sum\sb{j=1}\sp{n}
\sum\sb{\pm}
\bar\psi\sb{\xxd\pm\e\sb j}\sp{\ttd}
\cdot(\psi\sb{\xxd}\sp{\ttd+1}+\psi\sb{\xxd}\sp{\ttd-1})
-\mathrm{C.\,C.\,}
\nonumber
\\
&=&
n\sum\sb{\xxd\in\Z^n}
\big(1+\tau^2 B\sb{\xxd}(\abs{\psi\sb{\xxd}\sp{\ttd+1}}^2,\abs{\psi\sb{\xxd}\sp{\ttd-1}}^2)\big)
\abs{\psi\sb{\xxd}\sp{\ttd+1}+\psi\sb{\xxd}\sp{\ttd-1}}^2
-\mathrm{C.\,C.\,}
=0.
\nonumber
\end{eqnarray}
To get to the second line,
we used the complex conjugate of \eqref{dkg-fn}.
This finishes the proof of Theorem~\ref{theorem-charge}.
\end{proof}

\section{Proof of well-posedness and uniqueness results}
\label{sect-eu}

\begin{proof}[Proof of Theorem~\ref{theorem-e}]
We rewrite equation \eqref{dkg-c}
in the following form:
\begin{equation}\label{dkg-f}
\big(\psi\sb{\xxd}\sp{\ttd+1}+\psi\sb{\xxd}\sp{\ttd-1}\big)
\big(1
+\tau^2
B\sb{\xxd}(\abs{\psi\sb{\xxd}\sp{\ttd+1}}^2,\abs{\psi\sb{\xxd}\sp{\ttd-1}}^2)
\big)
=
\frac{\tau^2}{\varepsilon^2}
\sum\sb{j=1}\sp{n}
\big(
\psi\sb{\xxd+\e\sb j}\sp{\ttd}-2\psi\sb{\xxd}\sp{\ttd}+\psi\sb{\xxd-\e\sb j}\sp{\ttd}
\big)
+2\psi\sb{\xxd}\sp{\ttd},
\quad
\xxd\in\Z^n,
\quad
\ttd\in\Z.
\end{equation}
By \eqref{def-k1}
and the choice of $\tau\sb{1}$ in Theorem~\ref{theorem-e},
for $\tau\in(0,\tau\sb{1})$
one has
\begin{equation}\label{w-pos}
\inf\sb{\xxd\in\Z^n,\lambda\ge 0}
\big(1+\tau^2\p\sb\lambda V\sb{\xxd}(\lambda)\big)>0.
\end{equation}
Since
\begin{equation}\label{inf-b-w}
\inf\sb{\xxd\in\Z^n,\,\lambda\ge 0,\,\mu\ge 0}B\sb{\xxd}(\lambda,\mu)
=
\inf\sb{\xxd\in\Z^n,\,\lambda\ge 0,\,\mu\ge 0,\,\lambda\ne\mu}\frac{V\sb{\xxd}(\lambda)-V\sb{\xxd}(\mu)}{\lambda-\mu}
=
\inf\sb{\xxd\in\Z^n,\,\lambda\ge 0}\p\sb\lambda V\sb{\xxd}(\lambda),
\end{equation}
inequality \eqref{w-pos}
yields
\begin{equation}\label{1e2b}
c:=\inf\sb{\xxd\in\Z^n,\lambda\ge 0,\mu\ge 0}
\big(1+\tau^2 B\sb{\xxd}(\lambda,\mu)\big)
>0.
\end{equation}

Let us show that
equation \eqref{dkg-f} allows us to find $\psi\sb{\xxd}\sp{\ttd+1}$,
for any given $\xxd\in \Z^n$ and $\ttd\in\Z$,
once one knows $\psi\sp{\ttd}$ and $\psi\sp{\ttd-1}$.
Equation \eqref{dkg-f} implies that
\begin{equation}\label{eqn-on-xi}
\big(1+\tau^2 B\sb{\xxd}(\abs{\psi\sb{\xxd}\sp{\ttd+1}}^2,\abs{\psi\sb{\xxd}\sp{\ttd-1}}^2)\big)
(\psi\sb{\xxd}\sp{\ttd+1}+\psi\sb{\xxd}\sp{\ttd-1})
=\xi\sb{\xxd}\sp{\ttd},
\end{equation}
\begin{equation}\label{def-xi}
\xi\sb{\xxd}\sp{\ttd}:=
\frac{\tau^2}{\varepsilon^2}
\sum\sb{j=1}\sp{n}
\big(\psi\sb{\xxd+\e\sb j}\sp{\ttd}-2\psi\sb{\xxd}\sp{\ttd}+\psi\sb{\xxd-\e\sb j}\sp{\ttd}\big)
+2\psi\sb{\xxd}\sp{\ttd}\in\C^N.
\end{equation}
If $\xi\sb{\xxd}\sp{\ttd}=0$, then there is a solution to \eqref{eqn-on-xi}
given by
$\psi\sb{\xxd}\sp{\ttd+1}=-\psi\sb{\xxd}\sp{\ttd-1}$.
Due to \eqref{1e2b},
this solution is unique.
Now let us assume that $\xi\sb{\xxd}\sp{\ttd}\ne 0$.
We see from \eqref{eqn-on-xi}
that we are to have
\begin{equation}\label{xi-eta-zeta}
\psi\sb{\xxd}\sp{\ttd+1}+\psi\sb{\xxd}\sp{\ttd-1}=s \xi\sb{\xxd}\sp{\ttd},
\qquad
\mbox{ with some $s\in\R$}.
\end{equation}
Let us introduce  the function
\begin{equation}\label{def-fs}
f(s)
:=\big(1+\tau^2 B\sb{\xxd}(\abs{s \xi\sb{\xxd}\sp{\ttd}-\psi\sb{\xxd}\sp{\ttd-1}}^2,\abs{\psi\sb{\xxd}\sp{\ttd-1}}^2)
\big)s.
\end{equation}
We do not indicate dependence of
$f$ on $\psi\sb{\xxd}\sp{\ttd-1}$, $\xi\sb{\xxd}\sp{\ttd}$, and $\xxd$,
treating them as parameters.
For $\xi\sb{\xxd}\sp{\ttd}\ne 0$,
we can solve \eqref{eqn-on-xi}
if we can find $s\in\R$ such that
\begin{equation}\label{eqn-on-s}
f(s)=1.
\end{equation}
Since $f(0)=0$,
while
$
\lim\sb{s\to\infty}
f(s)=+\infty
$
by \eqref{1e2b},
one concludes that there
is at least one solution
$s>0$
to \eqref{eqn-on-s}.

Let us prove that
once
$(\psi\sp{0},\psi\sp{1})\in l\sp 2(\Z^n)\times l\sp 2(\Z^n)$,
then one also knows that
$\norm{\psi\sp{\ttd}}\sb{l\sp 2(\Z^n)}$
remains finite
(but not necessarily uniformly bounded)
for all $\ttd\in\Z$.
As it follows from \eqref{1e2b} and \eqref{eqn-on-xi},
\begin{equation}\label{psi-recurrent}
\abs{\psi\sb{\xxd}\sp{\ttd+1}}
\le
\frac{1}{c}
\abs{\xi\sb{\xxd}\sp{\ttd}}
+\abs{\psi\sb{\xxd}\sp{\ttd-1}}.
\end{equation}
Since
$\norm{\xi\sp{\ttd}}\sb{l\sp 2(\Z^n)}
\le
\left(\frac{4\tau^2}{\varepsilon^2}+2\right)
\norm{\psi\sp{\ttd}}\sb{l\sp 2(\Z^n)}
$
by \eqref{def-xi},
the relation
\eqref{psi-recurrent}
implies the estimate
\begin{equation}
\norm{\psi\sp{\ttd+1}}\sb{l\sp 2(\Z^n)}
\le
\frac{1}{c}
\left(\frac{4\tau^2}{\varepsilon^2}+2\right)
\norm{\psi\sp{\ttd}}\sb{l\sp 2(\Z^n)}
+\norm{\psi\sp{\ttd-1}}\sb{l\sp 2(\Z^n)},
\end{equation}
and, by recursion,
the finiteness of $\norm{\psi\sp{\ttd}}\sb{l\sp 2(\Z^n)}$
for all $\ttd\ge 0$.
The case $\ttd\le 0$ is finished in the same way.
\end{proof}

Now we turn to the uniqueness of solutions
to the Cauchy problem for equation \eqref{dkg-c}.

\begin{proof}[Proof of Theorem~\ref{theorem-u}]
First, note that,
by Remark~\ref{remark-k2-k3},
$\tau\sb{1}$ from Theorem~\ref{theorem-e}
and $\tau\sb{2}$ from Theorem~\ref{theorem-u}
are related by
$\tau\sb{2}\le\tau\sb{1}$.
Therefore, the existence of a solution
$\psi\sb{\xxd}\sp{\ttd}$ to
the Cauchy problem for equation \eqref{dkg-c}
follows from Theorem~\ref{theorem-e}.

Let us prove that this solution
$\psi\sb{\xxd}\sp{\ttd}$ is unique.
When in \eqref{def-xi}
one has
\[
\xi\sb{\xxd}\sp{\ttd}:=\frac{\tau^2}{\varepsilon^2}\sum\sb{j=1}\sp{n}
(\psi\sb{\xxd+\e\sb{j}}\sp{\ttd}-2\psi\sb{\xxd}\sp{\ttd}+\psi\sb{\xxd-\e\sb{j}}\sp{\ttd})
+2\psi\sb{\xxd}\sp{\ttd}=0,
\]
then,
by \eqref{1e2b},
the only solution $\psi\sb{\xxd}\sp{\ttd+1}$
to \eqref{eqn-on-xi}
is given by $\psi\sb{\xxd}\sp{\ttd+1}=-\psi\sb{\xxd}\sp{\ttd-1}$.
We now consider the case $\xi\sb{\xxd}\sp{\ttd}\ne 0$.
By \eqref{eqn-on-xi}, \eqref{xi-eta-zeta},
and \eqref{def-fs},
it suffices to prove the uniqueness of the solution
to \eqref{eqn-on-s}.
This will follow if we show that
$f(s)$ satisfies
\begin{equation}\label{fp}
f'(s)>0,
\qquad
s\in\R.
\end{equation}
The explicit expression for $f'(s)$ is
\begin{equation}\label{def-fp-0}
1
+\tau^2 B\sb{\xxd}(\abs{s \xi\sb{\xxd}\sp{\ttd}-\psi\sb{\xxd}\sp{\ttd-1}}^2,\abs{\psi\sb{\xxd}\sp{\ttd-1}}^2)
+\tau^2 \p\sb\lambda B\sb{\xxd}(\abs{s \xi\sb{\xxd}\sp{\ttd}-\psi\sb{\xxd}\sp{\ttd-1}}^2,\abs{\psi\sb{\xxd}\sp{\ttd-1}}^2)
(-2\Re(\bar\psi\sb{\xxd}\sp{\ttd-1}\cdot\xi\sb{\xxd}\sp{\ttd})+2\abs{\xi\sb{\xxd}\sp{\ttd}}^2 s)s.
\end{equation}
Using the relation \eqref{xi-eta-zeta},
we derive the identity
\[
(-2\Re(\bar\psi\sb{\xxd}\sp{\ttd-1}\cdot\xi\sb{\xxd}\sp{\ttd})+2\abs{\xi\sb{\xxd}\sp{\ttd}}^2 s)s
=2\Abs{s \xi\sb{\xxd}\sp{\ttd}-\frac{\psi\sb{\xxd}\sp{\ttd-1}}{2}}^2-\frac{\abs{\psi\sb{\xxd}\sp{\ttd-1}}^2}{2}
=2\Abs{\psi\sb{\xxd}\sp{\ttd+1}+\frac{\psi\sb{\xxd}\sp{\ttd-1}}{2}}^2-\frac{\abs{\psi\sb{\xxd}\sp{\ttd-1}}^2}{2}
\]
and rewrite the expression \eqref{def-fp-0} for $f'(s)$ as
\begin{equation}
f'(s)
=
1+\tau^2
\Big[
B\sb{\xxd}(\abs{\psi\sb{\xxd}\sp{\ttd+1}}^2,\abs{\psi\sb{\xxd}\sp{\ttd-1}}^2)
+2\p\sb\lambda B\sb{\xxd}(\abs{\psi\sb{\xxd}\sp{\ttd+1}}^2,\abs{\psi\sb{\xxd}\sp{\ttd-1}}^2)
\Big(
\abs{\psi\sb{\xxd}\sp{\ttd+1}+\frac{\psi\sb{\xxd}\sp{\ttd-1}}{2}}^2
-\frac{\abs{\psi\sb{\xxd}\sp{\ttd-1}}^2}{4}
\Big)
\Big].
\label{def-fp-1}
\end{equation}
We denote $\lambda=\abs{\psi\sb{\xxd}\sp{\ttd+1}}^2$,
$\mu=\abs{\psi\sb{\xxd}\sp{\ttd-1}}^2$.
Since
$
\lambda-\sqrt{\lambda\mu}+\frac{\mu}{4}
\le
\abs{\psi\sb{\xxd}\sp{\ttd+1}+\frac{\psi\sb{\xxd}\sp{\ttd-1}}{2}}^2
\le\lambda+\sqrt{\lambda\mu}+\frac{\mu}{4},
$
we see that
\begin{equation}\label{def-fp-2}
f'(s)\ge 1+\tau^2\min\sb{\pm}
\inf\sb{\xxd\in\Z^n,\,\lambda\ge 0,\,\mu\ge 0}K\sb{\xxd}\sp\pm(\lambda,\mu),
\quad
\mbox{with}
\quad
K\sb{\xxd}\sp\pm(\lambda,\mu)
=
B\sb{\xxd}(\lambda,\mu)
+2\p\sb\lambda B\sb{\xxd}(\lambda,\mu)
(\lambda\pm\sqrt{\lambda\mu}).
\end{equation}

By
\eqref{def-k2}
and 
by our choice of $\tau\sb{2}$
in Theorem~\ref{theorem-u},
for any $\tau\in(0,\tau\sb{2})$
we have
\[
\kappa:=\inf\sb{\xxd\in\Z^n,\,\lambda\ge 0,\,\mu\ge 0}
\left\{1+\tau^2 K\sb{\xxd}\sp\pm(\lambda,\mu)\right\}>0;
\]
then,
by \eqref{def-fp-2},
$f'(s)\ge \kappa$,
where $\kappa>0$.
It follows that
for $\xi\sb{\xxd}\sp{\ttd}\ne 0$
there is a unique $s$
which solves \eqref{eqn-on-s}.
Hence, there is a unique
solution
$\psi\sb{\xxd}\sp{\ttd+1}$
to equation \eqref{eqn-on-xi}
for given values $\psi\sb{\xxd}\sp{\ttd-1}$
and
$\xi\sb{\xxd}\sp{\ttd}$.
This finishes the proof of the Theorem.
\end{proof}

\begin{proof}[Proof of Theorem~\ref{theorem-pol}]
Let us prove that the condition
\eqref{def-k2} in Theorem~\ref{theorem-u}
is satisfied by any polynomial potential
of the form \eqref{poly}.
The inequality \eqref{def-k2}
will be satisfied if the highest order
term from $V(\lambda)$
contributes a strictly positive expression.
More precisely, we need to prove the following result.

\begin{lemma}\label{lemma-31}
Let $V(\lambda)=\lambda^{p+1}$,
so that
$B(\lambda,\mu)=\frac{\lambda^{p+1}-\mu^{p+1}}{\lambda-\mu}$,
$p\ge 0$.
Then the following inequality takes place:
\begin{equation}\label{wn}
\inf\sb{\lambda\ge 0,\,\mu\ge 0,\,\lambda^2+\mu^2=1}
\Big[
B(\lambda,\mu)
+2\p\sb\lambda B(\lambda,\mu)(\lambda\pm\sqrt{\lambda\mu})
\Big]>0.
\end{equation}
\end{lemma}

\begin{proof}
Since $B$ and $\p\sb\lambda B$
are strictly positive
for $\lambda^2+\mu^2>0$,
the inequality \eqref{wn}
is nontrivial
only for the negative sign in \eqref{wn}
and only when $\mu>\lambda$.
First we  note that
\[
B(\lambda,\mu)=\frac{\mu^{p+1}-\lambda^{p+1}}{\mu-\lambda},
\quad
\p\sb\lambda B(\lambda,\mu)
=
\frac{-(p+1)\lambda^p(\mu-\lambda)-\lambda^{p+1}+\mu^{p+1}}
{(\mu-\lambda)^2}
=
\frac{\mu^{p+1}-(p+1)\lambda^p\mu+p\lambda^{p+1}}
{(\mu-\lambda)^2}.
\]
Let $z\ge 0$ be such that
$z^2=\lambda/\mu$.
To prove the lemma, we need to check that 
\begin{equation}\label{nee}
\frac{1-z^{2p+2}}{1-z^2}
+2\frac{1-(p+1)z^{2p}+pz^{2p+2}}{(1-z^2)^2}
(z^2-z)>0,
\qquad
0\le z<1,
\end{equation}
or equivalently,
\[
(1+z)(1-z^{2p+2})-2z\big(1-(p+1)z^{2p}+pz^{2p+2}\big)>0.
\]
The left-hand side takes the form
\[
(1+z)(1-z^{2p+2})-2z\big(1-z^{2p+2}-(p+1)(z^{2p}-z^{2p+2})\big)
=(1-z)(1-z^{2p+2})
+2z(p+1)(z^{2p}-z^{2p+2}),
\]
which is clearly strictly positive
for all $0\le z<1$ and $p\ge 0$,
proving \eqref{nee}.
\end{proof}

This finishes the proof of
the first part of Theorem~\ref{theorem-pol};
now we turn to the second part.

\begin{lemma}[Uniqueness criterion]
\label{lemma-u-0}
Assume that for
a particular $\tau>0$
and for all $\lambda\ge 0$, $\mu\ge 0$, $\xxd\in\Z^n$,
the following inequalities hold:
\begin{equation}\label{cond-on-b}
1+
\tau^2
\inf\sb{\xxd\in\Z^n,\,\lambda\ge 0,\,\mu\ge 0}
\Big(B\sb{\xxd}(\lambda,\mu)-\p\sb\lambda B\sb{\xxd}(\lambda,\mu)\frac{\mu}{2}\Big)
>0;
\end{equation}
\begin{equation}\label{pb-positive}
\inf\sb{\xxd\in\Z^n,\,\lambda\ge 0,\,\mu\ge 0}
\p\sb\lambda B\sb{\xxd}(\lambda,\mu)\ge 0.
\end{equation}
Then the is a solution $\psi\sb{\xxd}\sp{\ttd}$
to the Cauchy problem for equation \eqref{dkg-c}
with arbitrary initial data $(\psi\sp{0},\psi\sp{1})$, 
and this solution is unique.
\end{lemma}

\begin{proof}[Proof of Lemma~\ref{lemma-u-0}]
The inequalities
\eqref{cond-on-b} and \eqref{pb-positive}
lead to
\[
1+\tau^2\inf\sb{\xxd\in\Z^n,\,\lambda\ge 0}
B\sb{\xxd}(\lambda,\lambda)>0,
\]
hence,
by the same argument as in Theorem~\ref{theorem-e},
there is a solution $\psi\sb{\xxd}\sp{\ttd}$.
The relation
\eqref{def-fp-1}
shows that $f'(s)\ge c$ for some $c>0$.
The rest of the proof is the same as for Theorem~\ref{theorem-u}.
\end{proof}

In the second part of Theorem~\ref{theorem-pol},
we assume that
\begin{equation}
V\sb{\xxd}(\lambda)=\sum\sb{q=0}\sp{4}C\sb{\xxd,q}\lambda^{q+1},
\qquad
\xxd\in\Z^n,\quad\lambda\ge 0,
\end{equation}
where
$C\sb{\xxd,q}\ge 0$ for $\xxd\in\Z^n$ and $1\le q\le 4$,
and
\begin{equation}\label{def-k3-1}
k\sb 3=
\inf\sb{\xxd\in\Z^n}
C\sb{\xxd,0}>-\infty.
\end{equation}
One can see that
the term $C\sb{\xxd,0}\lambda$ in $V\sb{\xxd}(\lambda)$
contributes to $B\sb{\xxd}(\lambda,\mu)$
the expression
$b\sb{\xxd,0}(\lambda,\mu)=C\sb{\xxd,0}$,
while each term
in $V\sb{\xxd}(\lambda)$
of the form
$
C\sb{\xxd,q}\lambda^{q+1},
$
with
$1\le q\le 4$
and $C\sb{\xxd,q}\ge 0$,
contributes to $B\sb{\xxd}(\lambda,\mu)$
the expression
$C\sb{\xxd,q} b\sb{q}(\lambda,\mu)$,
with
$b\sb{q}(\lambda,\mu)=\sum\sb{k=0}\sp{q}\lambda^{q-k}\mu^k$.
For $\tau\in(0,\tau\sb{3})$,
with $\tau\sb{3}=\sqrt{-1/k\sb 3}$
for $k\sb 3<0$
and $\tau\sb{3}=+\infty$
for $k\sb 3\ge 0$,
one has
\begin{equation}\label{1ec}
1+\tau^2\inf\sb{\xxd\in\Z^n}C\sb{\xxd,0}>0.
\end{equation}

\begin{lemma}
\label{lemma-four}
For $1\le q\le 4$,
$b\sb{q}(\lambda,\mu)=\sum\sb{k=0}\sp{q}\lambda^{q-k}\mu^k$
satisfies the inequality
\[
b\sb{q}(\lambda,\mu)
\ge\p\sb\lambda b\sb{q}(\lambda,\mu)\frac{\mu}{2}
\qquad
\mbox{for all \ $\lambda,\,\mu\ge 0$.}
\]
\end{lemma}

By \eqref{1ec} and Lemma~\ref{lemma-four},
condition
\eqref{cond-on-b}
is satisfied.
Since $C\sb{\xxd,q}\ge 0$ for $1\le q\le 4$,
each term
$C\sb{\xxd,q}b\sb q(\lambda,\mu)$
satisfies
condition
\eqref{pb-positive}.
Therefore, by Lemma~\ref{lemma-u-0},
there is a unique solution
$\psi\sb{\xxd}\sp{\ttd}$ to the Cauchy problem for equation \eqref{dkg-c}.
This finishes the proof of Theorem~\ref{theorem-pol}.
\end{proof}

\section{Conclusion}
We found out that
the Strauss -- Vazquez finite-difference scheme
\cite{MR0503140}
for
the $\mathbf{U}(1)$-invariant
nonlinear wave equation in $n$ spatial dimensions,
with the grid ratio
$\tau/\varepsilon\le 1/\sqrt{n}$,
admits the positive-definite discrete analog of the energy
which is conserved.
The result holds in any spatial dimension $n\ge 1$,
for the field valued in $\C^N$, $N\ge 1$.
In the case of the nonlinear Klein-Gordon equation
with positive potential,
this provides a priori bounds for the solution,
showing that the finite-difference scheme is stable.

We found out that
if the grid ratio is $\tau/\varepsilon=1/\sqrt{n}$,
then this finite-difference scheme
also preserves the discrete charge.

We proved that the solution of the corresponding Cauchy problem
exists and is unique
for a broad class of nonlinearities.
In particular, this is the case
for any confining polynomial potential
if the discretization is sufficiently small.
Finally,
we indicated a class of polynomials
for which the size of the discretization
could be readily specified.

\bigskip

\noindent
ACKNOWLEDGMENTS.
The authors are grateful to
Juliette Chabassier and Patrick Joly
for providing us with the references
and with 
their latest paper \cite{SubCMAME},
and to Sergey Pirogov for an important remark.

\hbadness=10000


\def\cprime{$'$} \def\cprime{$'$} \def\cprime{$'$} \def\cprime{$'$}
  \def\cprime{$'$} \def\cprime{$'$} \def\cprime{$'$} \def\cprime{$'$}
  \def\cprime{$'$} \def\cprime{$'$} \def\cprime{$'$}

\end{document}